\documentclass[11pt,reqno]{amsart}
\usepackage{a4wide}
\usepackage[utf8x]{inputenc}
\usepackage[T1]{fontenc}
\usepackage[english]{babel}
\usepackage{hyperref}
\usepackage{amsfonts,amsmath,amssymb,amsthm,amsrefs}
\usepackage{latexsym}
\usepackage{layout}
\usepackage{dsfont}
\usepackage{xcolor}

\providecommand{\R}{}

\providecommand{\N}{}

\renewcommand{\R}{\mathbb{R}}

\renewcommand{\N}{{\mathbb N}}

\newcommand{\I}[1]{{\mathbf 1}_{\left\{#1\right\}}}									
\newcommand{\set}[1]{\left\{ #1 \right\}}								

\newcommand{\esub}[1]{{\mathbf E_{#1}}}

\newcommand\cF{\mathcal F}

\providecommand{\ora}[1]{}
\renewcommand{\ora}[1]{\overrightarrow{#1}}

\newcommand{\threepartdef}[6]
{
	\left\{
		\begin{array}{ll}
			#1 & \mbox{if } #2 \\
			#3 & \mbox{if } #4 \\ 
			#5 & \mbox{if } #6
		\end{array}
	\right.
}
\newcommand{\twopartdef}[4]
{
	\left\{
		\begin{array}{ll}
			#1 & \mbox{if } #2 \\
			#3 & \mbox{if } #4
		\end{array}
	\right.
}
\newtheorem{thm}{Theorem}
\newtheorem{lem}[thm]{Lemma}
\newtheorem{prop}[thm]{Proposition}
\newtheorem{cor}[thm]{Corollary}

\newtheorem{rem}{Remark}
\numberwithin{thm}{section}
\numberwithin{equation}{section}
\begin{document}

\title[De Finetti's control problem]{De Finetti's control problem with a concave bound on the control rate}

\author[Locas]{F\'elix Locas}
\address{D\'epartement de math\'ematiques, Universit\'e du Qu\'ebec \`a Montr\'eal (UQAM), 201 av.\ Pr\'esident-Kennedy, Montr\'eal (Qu\'ebec) H2X 3Y7, Canada}
\email{locas.felix@uqam.ca, renaud.jf@uqam.ca}

\author[Renaud]{Jean-Fran\c cois Renaud}


\date{\today}



\begin{abstract}
We consider De Finetti's control problem for absolutely continuous strategies with control rates bounded by a concave function and prove that a generalized mean-reverting strategy is optimal. In order to solve this problem, we need to deal with a nonlinear Ornstein-Uhlenbeck process. Despite the level of generality of the bound imposed on the rate, an explicit expression for the value function is obtained up to the evaluation of two functions.This optimal control problem has those with control rates bounded by a constant and a linear function, respectively, as special cases.
\end{abstract}

\maketitle

\section{Introduction}

Nowadays, \textit{De Finetti's control problem} refers to a family of stochastic optimal control problems concerned with the maximization of withdrawals made from a stochastic system. While it has interpretations in models of population dynamics and natural resources and in inventory models,  it originates from the field of insurance mathematics. Usually, the performance function is the expected time-discounted value of all withdrawals made up to a first-passage stopping time. In the original insurance context, it is interpretated as follows: find the optimal way to pay out dividends, taken from the insurance surplus process, until ruin is declared. In this case, the performance function is the expectation of the total amount of discounted dividend payments made up to the time of ruin. Therefore,  the difficulty consists in finding the optimal balance between paying out dividends as much (and as early) as possible while avoiding ruin to maintain those payments on the long run. Similar interpretations can be made in population dynamics (harvesting) and for natural resources extractions; see, e.g., \cites{alvarez-shepp_1998, ekstrom-lindensjo_2021}. In any case, the overall objective is the identification of the optimal way to withdraw from the system (optimal strategy) and the derivation of an analytical expression for the optimal value function. 

In this paper, we consider absolutely continuous strategies. Except for very recent contributions (see \cites{albrecher-azcue-muler_2022,  angoshtari-bayraktar-young_2019, renaud-simard_2021}), most of the literature has considered constant bounds on the control rates (see, e.g., \cites{asmussen-taksar_1997,jeanblanc-shiryaev_1995, choulli-et-al_2004}). Following the direction originally taken by \cite{alvarez-shepp_1998} and later on by \cite{renaud-simard_2021}, we apply a much more general bound, namely a concave function applied to the current level of the system.


\subsection{Model and problem formulation}

Let $(\Omega, \cF, (\cF_t)_{t \geq 0}, \mathbf{P})$ be a filtered probability space. Fix $\mu, \sigma > 0$. The state process $X = \set{X_t, t \geq 0}$ is given by
\begin{equation}\label{noncontrol}
X_t = x + \mu t + \sigma W_t,
\end{equation}
where $x \in \R$ and where $W = \set{W_t, t \geq 0}$ is an $(\cF_t)$-adapted standard Brownian motion.  For example, the process $X$ can be interpreted as the (uncontrolled) density of a population or the (uncontrolled) surplus process of an insurance company. As alluded to above, we are interested in absolutely continuous control processes. More precisely, a control strategy $\pi$ is characterized by a nonnegative and adapted control rate process $l^\pi = \set{l^\pi_t, t \geq 0}$ yielding the cumulative control process
\begin{equation*}
L_t^\pi = \int_0^t l_s^\pi \mathrm{d}s .
\end{equation*}
Note that $L^\pi = \set{L^\pi_t, t \geq 0}$ is nondecreasing and such that $L^\pi_0=0$. The corresponding controlled process $U^\pi = \set{U^\pi_t, t \geq 0}$ is then given by
\begin{equation*}
U_t^\pi = X_t - L_t^\pi.
\end{equation*}

In the above mentioned applications (e.g., harvesting and dividend payments), it makes sense to allow for higher rates when the underlying state process is far from its critical level and to allow for higher (relative) increase in this rate as early as possible. The situation is reminiscent of utility functions in economics. In this direction, let us fix an increasing and concave function $F \colon \R \to \R$ such that $F(0) \geq 0$. For technical reasons, we assume $F$ is a differentiable Lipschitz function. Finally, a strategy $\pi$ is said to be admissible if its control rate is also such that
\begin{equation}\label{eq:concave-bound}
0 \leq l_t^\pi \leq F(U_t^\pi),
\end{equation}
for all $0 \leq t \leq \tau_0^\pi$, where $\tau_0^\pi = \inf\set{t > 0 \colon U_t^\pi < 0}$ is the termination time.  The termination level $0$ is chosen only for simplicity. Let $\Pi^F$ be the set of all admissible strategies. 

From now, we will write $\mathbf{P}_x$ for the probability measure associated with the starting point $x$ and $\mathbf{E}_x$ for the expectation with respect to $\mathbf{P}_x$. When $x=0$, we write $\mathbf{P}$ and $\mathbf{E}$. 

Fix a time-preference parameter $q > 0$ and then define the value of a strategy $\pi \in \Pi^F$ by
\begin{equation*}
V_\pi(x) = \mathbf{E}_x\left[\int_0^{\tau_0^\pi} e^{-qt} l_t^\pi \mathrm{d}t \right], \quad x \geq 0.
\end{equation*}
Note that in our model, we have $V_\pi(0)=0$, for all $\pi \in \Pi^F$.

We want to find an optimal strategy, that is a strategy $\pi^\ast \in \Pi^F$ such that, for all $x \geq 0$ and for all $\pi \in \Pi^F$, we have $V_{\pi^*}(x) \geq V_\pi(x)$. We also want to compute the optimal value function given by
\begin{equation*}
V(x) = \sup_{\pi \in \Pi^F} V_\pi(x), \quad x \geq 0 .
\end{equation*}

\begin{rem}\label{remark}
The optimal value function and an optimal strategy (if it exists) should be independent of the behaviour of $F$ when $x < 0$. The Lipschitz condition assumption is a sufficient condition for the existence and the uniqueness of a (strong) solution to the stochastic differential equation (SDE) defined in~\eqref{allin}; see, e.g., \cite{oksendal_2003}. In fact, any increasing differentiable concave function $F : (0, +\infty) \to \R$ such that $F(0) \geq 0$ and $F'(0+) < \infty$ satisfy the conditions for our model: we can extend $F$ as a Lipschitz continuous function by extending $F$, such that $F(x) = F'(0+)x + F(0)$, for all $x \leq 0$.
\end{rem}

\subsection{More related literature}

De Finetti's optimal control problems are adaptations and interpretations of Bruno de Finetti's seminal work \cite{definetti_1957}. As they have been extensively studied in a variety of models, especially over the last 25 years, it is nearly impossible to provide a proper literature review on this topic in an introduction. See the review paper \cite{albrecher-thonhauser_2009} for more details. Therefore, let us focus on the literature closer to our model and our problem.

In \cites{jeanblanc-shiryaev_1995, asmussen-taksar_1997}, the problem for absolutely continuous strategies in a Brownian model is tackled under the following assumption on the control rate:
\begin{equation}\label{eq:constant-bound}
0 \leq l_t^\pi \leq R ,
\end{equation}
where $R>0$ is a given constant; see also \cite{kyprianou-et-al_2012} for the same problem in a Lévy model. For this problem, a threshold strategy is optimal: above the optimal barrier level, the maximal control rate $R$ is applied, otherwise the system is left uncontrolled. In other words, a threshold strategy is characterized by a barrier level and the corresponding controlled process is a Brownian motion with a two-valued drift.  Very recently, in \cite{renaud-simard_2021}, the admissibility condition in~\eqref{eq:constant-bound} was replaced by the following one:
\begin{equation}\label{eq:linear-bound}
0 \leq l_t^\pi \leq K U^\pi_t ,
\end{equation}
where $K>0$ is a given constant. In other words, the control rate is now bounded by a linear transformation of the current state. Note that this idea had already been considered in a biological context (but under a different model) in \cite{alvarez-shepp_1998}. In \cite{renaud-simard_2021}, it is proved that a mean-reverting strategy is optimal: above the optimal barrier level, the maximal control rate $KU^\pi_t$ is applied, otherwise the system is left uncontrolled. Again, such a control strategy is characterized by a barrier level, but now the controlled process is a \textit{refracted} diffusion process, switching between a Brownian motion with drift and an Ornstein-Uhlenbeck process.

To the best of our knowledge, the concept of \textit{mean-reverting strategy} first appeared in \cite{avanzi-wong_2012}. It was argued that, in an insurance context, this type of strategies has desirable properties for shareholders. Then, in \cite{renaud-simard_2021}, the name \textit{mean-reverting strategies} was used for a larger family of control strategies; in fact, the mean-reverting strategy in \cite{avanzi-wong_2012} is one member (when the barrier level is equal to zero) of this sub-family of strategies. As mentioned above, a mean-reverting strategy is optimal in the case $F(x)=Kx$. Following those lines, we will use the name \textit{generalized mean-reverting strategy} for similar bang-bang strategies corresponding to the general case of an increasing and concave function $F$.

\subsection{Main results and outline of the paper}

Obviously, the bounds imposed on the control rates in these last two problems, as given in~\eqref{eq:constant-bound} for the constant case and in~\eqref{eq:linear-bound} for the linear case, are special cases of the one considered in our problem and presented in~\eqref{eq:concave-bound}. A solution to our general problem is given in Theorem~\ref{mainresult}. As we will see, an optimal control strategy is provided by a generalized mean-reverting strategy, which is also characterized by a barrier level $b$. More specifically, for such a strategy, the optimal controlled process is given by the following diffusion process:
\[
\mathrm{d}U_t^b = \left(\mu - F(U_t^b)\I{U_t^b > b}\right)\mathrm{d}t + \sigma \mathrm{d}W_t .
\]

We use a guess-and-verify approach: first, we compute the value of any generalized mean-reverting strategy (Proposition~\ref{valuefunctionprop}) using a Markovian decomposition and a perturbation approach; second, we find the optimal barrier level (Proposition~\ref{propsolution}); finally, using a Verification Lemma (Lemma~\ref{verificationlemma}), we prove that this \textit{best} mean-reverting strategy is in fact an optimal strategy for our problem. Despite the level of generality (of $F$),  our solution for this control problem is explicit up to the evaluation of two \textit{special functions} strongly depending on $F$: $\varphi_F$, the solution to an ordinary differential equation, and $I_F$, an expectation functional. While $\varphi_F$ is one of the two well-known fundamental solutions associated to a diffusion process (also the solution to a first-passage problem), the definition of $I_F$ (see Equation~\eqref{IF}) is an intricate expectation for which many properties can be obtained (see Proposition~\ref{associate}).

The rest of the paper is organized as follows. In Section 2, we provide important preliminary results on the functions at the core of our main result and we state a Verification Lemma for the maximization problem. In Section 3, we introduce the family of generalized mean-reverting strategies and compute their value functions. In Section 4, we state and prove the main result, which is a solution to the control problem, before looking at specific examples. A proof of the Verification Lemma is provided in the Appendix.

\section{Preliminary results}



First, let $G \colon [0, \infty) \to \R$ be an increasing and differentiable function. Consider the homogeneous ordinary differential equation (ODE)
\begin{equation}\label{ODEF}
\Gamma_G(f) := \frac{\sigma^2}{2} f''(x) + (\mu - G(x))f'(x) - qf(x) = 0 ,  \quad x > 0,
\end{equation}
and denote by $\psi_G$ (resp.\ $\varphi_G$) a positive increasing (resp.\ decreasing) solution to~\eqref{ODEF}, with $\psi_G(0) = \varphi_G(+\infty) = 0$. Under the additional conditions that $\varphi_G(0) = 1$ and $\psi_G'(0) = 1$, it is known that $\psi_G$ and $\varphi_G$ are uniquely determined. It is also known that $\psi_G,\varphi_G \in C^3[0, \infty)$.

We will write $\Gamma$, $\psi$ and $\varphi$ when $G \equiv 0$.  In particular, it is easy to verify that, for $x \geq 0$,
\begin{equation}\label{scale}
\psi(x) = \frac{\sigma^2}{\sqrt{\mu^2 + 2q \sigma^2}} \mathrm e^{-(\mu/\sigma^2)x} \sinh \left((x/\sigma^2) \sqrt{\mu^2 + 2q \sigma^2}\right).
\end{equation}

In what follows, we will not manipulate this explicit expression for $\psi$. Instead, we will use its analytical properties (see Lemma~\ref{AnalyticLemma} below).

\begin{rem}\label{remscale}
The expression for $\psi$ is proportional to the expression for $W^{(q)}$, the $q$-scale function of the Brownian motion with drift $X$, used in \cite{renaud-simard_2021}. For more details,  see \cite{kyprianou_2014}.
\end{rem}

The next lemma gives analytical properties of the functions $\psi$ and $\varphi_G$. See \cite{ekstrom-lindensjo_2021} for a complete proof. 
\begin{lem}\label{AnalyticLemma}
Let $G \colon [0, \infty) \to \R$ be an increasing and differentiable function. The functions $\psi$ and $\varphi_G$ have the following analytical properties:
\begin{itemize}
\item[(a)] $\psi$ is strictly increasing and strictly concave-convex with a unique inflection point $\hat{b} \in (0, \infty)$;
\item[(b)] $\varphi_G$ is strictly decreasing and strictly convex.
\end{itemize}
\end{lem}

The value of the inflection point of $\psi$ is known explicitly; see, e.g., Equation (4) in \cite{renaud-simard_2021}. Our analysis does not depend on this specific value.



Recall from~\eqref{noncontrol} that $X_t = x + \mu t + \sigma W_t$. Define $U = \left\lbrace U_t , t \geq 0 \right\rbrace$ by
\begin{equation}\label{allin}
\mathrm{d}U_t = (\mu - F(U_t)) \mathrm{d}t + \sigma \mathrm{d}W_t. 
\end{equation}
Under our assumptions on $F$, there exists a unique strong solution to this last stochastic differential equation. This is what we called a \textit{nonlinear Ornstein-Uhlenbeck process}.

Now, for $a \geq 0$, define the following first-passage stopping times:
\[
\tau_a = \inf\set{t > 0 \colon X_t = a} \quad \text{and} \quad \tau^F_a = \inf\set{t > 0 \colon U_t = a}.
\]

%
%

It is well known that, for $0 \leq x \leq b$, we have
\[
\esub{x}\left[\mathrm{e}^{-q \tau_b} \I{\tau_b < \tau_0} \right] = \frac{\psi(x)}{\psi(b)} ,
\]
and, for $x \geq 0$, we have
\[
\esub{x}\left[\mathrm{e}^{-q \tau^F_0} \I{\tau^F_0 < \infty}\right] = \varphi_F(x) ,
\]
since it is assumed $\varphi_F(0)=1$.


Finally, define
\begin{equation}\label{IF}
I_F(x) = \esub{x}\left[\int_0^\infty \mathrm e^{-qt} F(U_t) \mathrm{d}t \right] ,   \quad x \in \R.
\end{equation}
As alluded to above, the analysis of this functional is of paramount importance for the solution of our control problem. The main difficulty in computing this functional lies in the fact that the dynamics of $U$ also depend on $F$.


\begin{prop}\label{associate}
The function $I_F \colon \R \to \R$ is a twice continuously differentiable, increasing and concave solution to the following ODE:
\begin{equation}\label{nonhomogeneous}
\Gamma_F(f) = -F ,  \quad x > 0.
\end{equation}
Moreover, we have $0 \leq I_F'(x) \leq 1$, for all $x \geq 0$.
\end{prop}

\begin{proof}
Under our assumptions on $F$, it is known that there exists a twice continuously differentiable solution $f$ to~\eqref{nonhomogeneous}. Applying Ito's Lemma to $\mathrm e^{-qt} f(U_t)$ and taking expectations, we obtain
\begin{equation*}
\mathbf{E}_x\left[\mathrm e^{-qt} f(U_t)\right] = f(x) - \esub{x}\left[\int_0^t \mathrm e^{-qs} F(U_s) \mathrm{d}s \right]. 
\end{equation*}
Letting $t \to \infty$, we get $f(x) = I_F(x)$. 

In the rest of the proof, we will use the notation $U^{x}$ for the solution to
\begin{equation}\label{notation}
\mathrm{d}U_t^x = (\mu - F(U_t^x))\mathrm{d}t + \sigma \mathrm{d}W_t, \quad U_0^x = x.
\end{equation}
This is the dynamics given in~\eqref{allin}.

Fix $x < y$. Note that $U_0^x < U_0^y$ and define $\kappa = \inf\set{t > 0 \colon U_t^x = U_t^y}$. Since $U_t^x < U_t^y$ for all $0 \leq t \leq \kappa$ and $F$ is increasing, we have
\begin{align*}
I_F(y) &= \mathbf{E}\left[ \int_0^\infty \mathrm{e}^{-qt} F(U_t^{y}) \mathrm{d}t\right] \\
&= \mathbf{E}\left[ \int_0^\kappa \mathrm{e}^{-qt} F(U_t^{y}) \mathrm{d}t\right] + \mathbf{E}\left[ \int_\kappa^\infty \mathrm{e}^{-qt} F(U_t^{y}) \mathrm{d}t\right] \\
&\geq \mathbf{E}\left[ \int_0^\kappa \mathrm{e}^{-qt} F(U_t^{x}) \mathrm{d}t\right] + \mathbf{E}\left[ \int_\kappa^\infty \mathrm{e}^{-qt} F(U_t^{x}) \mathrm{d}t\right] \\
&= I_F(x) ,
\end{align*}
where we used the fact that, for $t \geq \kappa$, we have $U_t^x = U_t^y$. This proves that $I_F$ is increasing. 

Now, fix $x, y \in \R$ and $\lambda \in [0,1]$. Define $z = \lambda x + (1 - \lambda)y$ and
\begin{equation*}
Y_t^z = \lambda U_t^x + (1 - \lambda) U_t^y .
\end{equation*}
By linearity, we have
\begin{equation*}
\mathrm{d}Y_t^z = (\mu - l_t^{Y}) \mathrm{d}t + \sigma \mathrm{d}W_t,  \quad Y_0^z = z,
\end{equation*}
where $l_t^{Y} := \lambda F(U_t^x) + (1 - \lambda) F(U_t^y)$. Since $F$ is concave, we have that $l_t^Y \leq F(Y_t^z)$ for all $t \geq 0$. We want to prove that, almost surely, we have $Y_t^z \geq U_t^z$ for all $t \geq 0$.
First, for all $t \geq 0$, we have
\begin{equation}\label{difference}
Y_t^z - U_t^z = \int_0^t \left(F(U_s^z) - l_s^{Y}\right)\mathrm{d}s.
\end{equation}
Note that, by the concavity of $F$, since $U_0^z = Y_0^z = z$, we have $F(U_0^z) \geq l_0^{Y}$. Now, define
\begin{equation*}
\tau = \inf\set{t > 0 \colon F(U_t^z) > l_t^{Y}} .
\end{equation*}
On $\set{\tau = +\infty}$, we have $Y_t^z = U_t^z$ for all $t$. On $\set{\tau < \infty}$, since the mapping $s \mapsto F(U_s^z) - l_s^{Y}$ is continuous almost surely, it follows from~\eqref{difference} that there exists $\epsilon > 0$ for which $Y_t^z > U_t^z$ for all $t \in ]\tau, \tau + \epsilon[$. But we can apply the same argument for another time point $s > \tau$: if there exists $s>\tau$ for which $Y_s^z = U_s^z$, then either we have $Y_t^z=U_t^z$ for all $t > s$, or there exists $s', \epsilon' > 0$, for which we have
\begin{equation*}
\twopartdef{Y_t^z = U_t^z}{s \leq t \leq s',}{Y_t^z > U_t^z}{s' < t < s'+\epsilon'.}
\end{equation*}
This proves that, for all $t \geq 0$, we have $Y_t^z \geq U_t^z$, which is equivalent to
\begin{equation*}
\int_0^t F(U_s^z) \mathrm{d}s \geq \int_0^t \left(\lambda F(U_s^x) + (1-\lambda) F(U_s^y)\right)\mathrm{d}s.
\end{equation*}
Since this is true for all $t \geq 0$, we further have that, for all $t \geq 0$,
\begin{equation}\label{justifier}
\int_0^t \mathrm e^{-qs} F(U_s^z) \mathrm{d}s \geq \int_0^t \mathrm e^{-qs}\left(\lambda F(U_s^x) + (1-\lambda) F(U_s^y)\right) \mathrm{d}s .
\end{equation}
Letting $t \to \infty$, it follows that
\begin{equation*}
I_F(\lambda x + (1-\lambda)y) \geq \lambda I_F(x) + (1-\lambda)I_F(y),
\end{equation*}
proving the concavity of $I_F$. 

Let $x, h \geq 0$ and define
\begin{equation*}
\kappa^h = \inf\set{t > 0 : U_t^{x+h} = U_t^x}.
\end{equation*}
We can write
\begin{multline*}
I_F(x+h) - I_F(x) = \mathbf{E} \left[ \int_0^\infty \mathrm e^{-qt} \left(F(U_t^{x+h}) - F(U_t^{x})\right) \mathrm{d}t \right] \\
= \mathbf{E} \left[\int_0^{\kappa^h} \mathrm e^{-qt} \left(F(U_t^{x+h}) - F(U_t^{x})\right) \mathrm{d}t \right] + \mathbf{E} \left[\int_{\kappa^h}^{\infty} \mathrm e^{-qt} \left(F(U_t^{x+h}) - F(U_t^{x})\right)\mathrm{d}t \right].
\end{multline*} 

If $\kappa^h$ is finite, the second expectation is zero, because for all $t \geq \kappa^h$ and for all $\omega \in \Omega$, we have $U_t^{x+h}(\omega) = U_t^{x}(\omega)$. Now, on $\set{\kappa^h < \infty}$, we have
\begin{equation*}
\int_0^{\kappa^h} \left(F(U_s^{x+h}) - F(U_s^{x}) \right)\mathrm{d}s = h,
\end{equation*}
which implies that
\begin{equation*}
\int_0^{\kappa^h} \mathrm e^{-qs} \left(F(U_s^{x+h}) - F(U_s^{x}) \right)\mathrm{d}s \leq  h.
\end{equation*}
On $\set{\kappa^h = \infty}$, we have
\begin{equation*}
\int_0^{t} \mathrm e^{-qs} \left(F(U_s^{x+h}) - F(U_s^{x}) \right)\mathrm{d}s < h,
\end{equation*}
for all $t \geq 0$, which implies that
\begin{equation*}
\int_0^{\infty} \mathrm e^{-qs} \left(F(U_s^{x+h}) - F(U_s^{x}) \right)\mathrm{d}s \leq h.
\end{equation*}
In conclusion, we have
\begin{equation*}
\mathbf{E}\left[\int_0^{\kappa^h} \mathrm e^{-qt} \left(F(U_t^{x+h}) - F(U_t^{x})\right) \mathrm{d}t \right] \leq h.
\end{equation*}
Letting $h \to 0$, we find that $I_F'(x) \leq 1$.
\end{proof}


Finally, here is a Verification Lemma for the stochastic control problem. The proof is given in Appendix A.
\begin{lem}\label{verificationlemma}
Let $\pi^* \in \Pi^F$ be such that $V_{\pi^*} \in C^2(0, \infty)$, $V_{\pi^*}'$ is bounded and, for all $x > 0$,
\begin{equation}\label{verificationsup}
\frac{\sigma^2}{2}V_{\pi^*}''(x) + \mu V_{\pi^*}'(x) - qV_{\pi^*}(x) + \sup_{0 \leq u \leq F(x)} u(1 - V_{\pi^*}'(x)) = 0.
\end{equation}
Then, $\pi^*$ is an optimal strategy. In this case, $V \in C^2(0, +\infty)$, $V'$ is bounded and $V$ satisfies~\eqref{verificationsup}. 
\end{lem}

\section{Generalized mean-reverting strategies}

Since the Hamilton-Jacobi-Bellman equation in~\eqref{verificationsup} is linear with respect to the control variable, we expect a bang-bang strategy to be optimal. Further, since from modelling reasons we expect the optimal value fonction $V$ to be concave, then an optimal strategy must be of the form
\begin{equation}\label{guess2}
l_s^\pi = \twopartdef{F(U_s^\pi)}{U_s^\pi > b,}{0}{U_s^\pi < b,}
\end{equation}
for some $b \geq 0$ to be determined.

%

Consequently, and following the line of reasoning in \cite{renaud-simard_2021}, let us define the family of \textit{generalized mean-reverting strategies}. For a fixed $b \geq 0$, define the generalized mean-reverting strategy $\pi_b$ and the corresponding controlled process $U^b := U^{\pi_b}$ by
\begin{equation}\label{u^b}
\mathrm{d}U_t^b = \left(\mu - F(U_t^b)\I{U_t^b > b}\right)\mathrm{d}t + \sigma \mathrm{d}W_t .
\end{equation} 
In other words,  the control rate $l^b := l^{\pi_b}$ associated to $\pi_b$ is given by
\[
l^b_t = F(U_t^b) \I{U_t^b > b} .
\]
Similarly, we define the corresponding value function by $V_b := V_{\pi_b}$.

Note that, if $b = 0$, then $U^0 = U$, with $U$ already defined in~\eqref{allin}.  


\begin{rem}
As discussed in Remark~\ref{remark}, there exists a unique strong solution to the SDE given in~\eqref{allin}. When $b>0$, the drift function is not necessarily continuous, but it can be shown, using for example the same steps as in \cite{renaud-simard_2021}, that a strong solution exists for~\eqref{u^b}. 
\end{rem}

The next proposition gives the value function of a generalized mean-reverting strategy.

\begin{prop}\label{valuefunctionprop}
The value function $V_0$ of the generalized mean-reverting strategy $\pi_0$ is continuously differentiable and given by
\begin{equation}\label{valuefunctionb=0}
V_0(x) = I_F(x) - I_F(0) \varphi_F(x), \quad x \geq 0.
\end{equation}
If $b > 0$, then the value function $V_b$ of the generalized mean-reverting strategy $\pi_b$ is continuously differentiable and given by
\begin{equation}\label{valuefunctionb>0}
V_b(x) = \twopartdef{C_1(b) \psi(x)}{0 \leq x \leq b,}{I_F(x) + C_2(b)\varphi_F(x)}{x \geq b,}
\end{equation}
where
\begin{align*}
C_1(b) &= \frac{I_F'(b) \varphi_F(b) - I_F(b) \varphi_F'(b)}{\psi'(b) \varphi_F(b) - \psi(b) \varphi_F'(b)} ,\\
C_2(b) &= \frac{I_F'(b) \psi(b) - I_F(b) \psi'(b)}{\psi'(b) \varphi_F(b) - \psi(b) \varphi_F'(b)}.
\end{align*}
\end{prop}

\begin{proof}
The proof follows the same steps as the one for Proposition~2.1 in \cite{renaud-simard_2021}.

Fix $b > 0$. Using the strong Markov property, we have, for $x \leq b$,
\begin{equation*}
V_b(x) = \esub{x}\left[\mathrm{e}^{-q \tau_b} \I{\tau_b < \tau_0} \right] V_b(b) .
\end{equation*}
Using the strong Markov property again, we get, for all $x > b$,
\begin{equation*}
V_b(x) = \esub{x}\left[ \int_0^\infty \mathrm{e}^{-qt} F(U_t) \mathrm{d}t \right] + \esub{x}\left[\mathrm{e}^{-q \tau^F_0} \I{\tau^F_0 < \infty}\right]\left(\frac{V_b(b) - \esub{b}\left[ \int_0^\infty \mathrm{e}^{-qt} F(U_t) \mathrm{d}t \right]}{\esub{b}\left[\mathrm{e}^{-q \tau^F_0} \I{\tau^F_0 < \infty}\right]}\right) .
\end{equation*}

Consequently, we can write
\begin{equation}\label{twopart}
V_b(x) = \twopartdef{\frac{\psi(x)}{\psi(b)} V_b(b)}{0 \leq x \leq b,}{I_F(x) + \frac{\varphi_F(x)}{\varphi_F(b)}(V_b(b) - I_F(b))}{x \geq b.}
\end{equation}

To conclude, we need to compute $V_b(b)$. For $n \in \N$ sufficiently large, consider the strategy $\pi_b^n$ consisting of using the maximal control rate $F(U_t^{\pi_b^n})$ when the controlled process is above $b$, until it goes below $b - 1/n$. We apply again the maximal control rate when the controlled process reaches $b$ again. Note that $\pi_b^n$ is admissible. We denote its value function by $V_b^n$. We can show that: \begin{equation*}
\lim_{n \to \infty} V_b^n(b) = V_b(b).
\end{equation*}

Using similar arguments as above, we can write
\begin{equation*}
V_b^n(b - 1/n) = \psi_b(b-1/n) V_b^n(b)
\end{equation*}
and
\begin{align*}
V_b^n(b) &= \esub{b}{\left[ \int_0^{\tau_{b-1/n}^{F}} e^{-qt} F(U_t^{\pi_b^n}) \mathrm{d}t\right]} + \esub{b}{\left[e^{-q \tau_{b-1/n}^{F}} \I{\tau_{b-1/n}^{F} < \infty}\right]}V_b^n(b-1/n) \\
&= I_F(b) + \frac{\varphi_F(b)}{\varphi_F(b-1/n)}\left(V_b^n(b-1/n) - I_F(b-1/n)\right).
\end{align*}
Solving for $V_b^n(b)$, we find
\begin{equation*}
V_b^n(b) = \frac{I_F(b-1/n)\varphi_F(b) - I_F(b) \varphi_F(b-1/n)}{\psi_b(b-1/n)\varphi_F(b) - \psi_b(b) \varphi_F(b-1/n)} ,
\end{equation*}
where $\psi_b(x)=\psi(x)/\psi(b)$.
Define
\begin{equation*}
G(y) = I_F(b-y)\varphi_F(b) - I_F(b) \varphi_F(b-y)
\end{equation*}
and
\begin{equation*}
H(y) = \psi_b(b-y) \varphi_F(b) - \psi_b(b) \varphi_F(b-y).
\end{equation*}
Note that $G(0) = H(0) = 0$. As $\psi_b, \varphi_F$ and $I_F$ are differentiable functions, dividing the numerator and the denominator by $1/n$ and taking the limit yields
\begin{equation*}
V_b(b) = \lim_{n \to \infty} V_b^n(b) = \frac{G'(0+)}{H'(0+)},
\end{equation*}
which leads to
\begin{equation}\label{vbb}
V_b(b) = \frac{I_F'(b) \varphi_F(b) - I_F(b) \varphi_F'(b)}{\psi_b'(b) \varphi_F(b) - \psi_b(b) \varphi_F'(b)}.
\end{equation}

Substituting~\eqref{vbb} into~\eqref{twopart}, we get:
\begin{equation}\label{closetoVb}
V_b(x) = \twopartdef{K_1(b) \psi_b(x)}{0 \leq x \leq b,}{I_F(x) + K_2(b) \varphi_F(x)}{x \geq b,}
\end{equation}
where
\begin{equation*}
K_1(b) = \frac{I_F'(b) \varphi_F(b) - I_F(b) \varphi_F'(b)}{\psi_b'(b) \varphi_F(b) - \psi_b(b) \varphi_F'(b)} \quad \text{and} \quad K_2(b) = \frac{I_F'(b) \psi_b(b) - I_F(b) \psi_b'(b)}{\psi_b'(b) \varphi_F(b) - \psi_b(b) \varphi_F'(b)}.
\end{equation*}

Using the definition of $\psi_b$, \eqref{closetoVb} can be rewritten as follows:
\begin{equation*}
V_b(x) = \twopartdef{C_1(b) \psi(x)}{0 \leq x \leq b,}{I_F(x) + C_2(b)\varphi_F(x)}{x \geq b.}
\end{equation*}

It is straightforward to check that $V_b \in C^1(0, +\infty)$, since elementary algebraic manipulations lead to $V_b'(b-) = V_b'(b+)$.

When $b=0$, using Markovian arguments as above, we can verify that
\begin{equation*}
V_0(x) = I_F(x) - I_F(0)\varphi_F(x), \quad x \geq 0,
\end{equation*}
from which it is clear that $V_0 \in C^1(0, +\infty)$.
\end{proof}

\section{Main results}

We are now ready to provide a solution to the general control problem. As mentioned before, depending on the set of parameters, an optimal strategy will be given by the generalized mean-reverting strategy $\pi_0$ or by a generalized mean-reverting strategy $\pi_{b^\ast}$, for a barrier level $b^\ast$ to be determined.


First, let us consider the situation in which the parameters are such that
\begin{equation}\label{premiercas}
I_F'(0) - I_F(0) \varphi_F'(0) \leq 1.
\end{equation}
Recalling from Proposition~\ref{valuefunctionprop}, that
\[
V_0(x) = I_F(x) - I_F(0) \varphi_F(x), \quad x \geq 0,
\]
we deduce that $V_0'(0) \leq 1$. 

Also, we have that $V_0$ is a concave function. Indeed,  using the notation introduced in~\eqref{notation}, we can write
\begin{equation*}
V_0(x) = \mathbf{E}\left[\int_0^{\tau_0^x} \mathrm e^{-qt} F(U_t^x) \mathrm{d}t\right] ,
\end{equation*}
where $\tau^x_0 = \inf\set{t > 0 \colon U^x_t = 0}$. Since the inequality in~\eqref{justifier} holds for all $t \geq 0$, it holds true also for the stopping time $\tau^x_0$. As a consequence, we have that $V_0$ is concave and we further have that $V_0'(x) \leq 1$ for all $x \geq 0$.

In conclusion, all the conditions of the Verification Lemma are satisfied and thus the generalized mean-reverting $\pi_0$ is an optimal strategy.


Now, let us consider the situation in which the parameters are such that
\begin{equation}\label{deuxiemecas}
I_F'(0) - I_F(0) \varphi_F'(0) > 1.
\end{equation}
From the Hamilton-Jacobi-Bellman equation~\eqref{verificationsup}, we expect an optimal barrier level $b^*$ to be given by
\begin{equation*}
V_{b^*}'(b^*) = 1 ,
\end{equation*}
which is equivalent to
\begin{equation}\label{solution}
I_F(b^*) - \frac{I_F'(b^*) \varphi_F(b^*)}{\varphi_F'(b^*)} = \frac{\psi(b^*)}{\psi'(b^*)} - \frac{\varphi_F(b^*)}{\varphi_F'(b^*)}.
\end{equation}

The next proposition is one of the most important results.
\begin{prop}\label{propsolution}
If $I_F'(0) - I_F(0)\varphi_F'(0) > 1$, then there exists a solution $b^* \in (0,\hat{b}]$ to~\eqref{solution}.
\end{prop}

\begin{proof}
Define
\begin{equation*}
g(y) = I_F(y) - \frac{I_F'(y) \varphi_F(y)}{\varphi_F'(y)}
\end{equation*}
and
\begin{equation*}
h(y) = \frac{\psi(y)}{\psi'(y)} - \frac{\varphi_F(y)}{\varphi_F'(y)}.
\end{equation*}
We see that $g(0) > h(0)$ is equivalent to $I_F'(0) - I_F(0) \varphi_F'(0) > 1$.  We will show that $g(\hat{b}) \leq h(\hat{b})$. The result will follow from the Intermediate Value Theorem. First, we have the following inequality: \begin{equation}\label{inequalityb}
\frac{\psi(\hat{b})}{\psi'(\hat{b})} - \frac{\varphi_F(\hat{b})}{\varphi_F'(\hat{b})} \geq \frac{F(\hat{b})}{q}.
\end{equation}
Indeed, by definition of $\hat{b}$ and $\psi$, we have $\frac{\psi(\hat{b})}{\psi'(\hat{b})} = \frac{\mu}{q}$. Also, by definition of $\varphi_F$, we have
\begin{equation*}
\frac{\sigma^2}{2} \frac{\varphi_F''(\hat{b})}{\varphi_F'(\hat{b})} + \mu - q \frac{\varphi_F(\hat{b})}{\varphi_F'(\hat{b})} - F(\hat{b}) = 0.
\end{equation*} 
Since $\varphi_F$ is convex and decreasing, \eqref{inequalityb} follows. 

Now, using Proposition~\ref{associate}, we can write
\begin{equation*}
I_F(\hat{b}) = \frac{\sigma^2}{2q}I_F''(\hat{b}) + \frac{\mu}{q} I_F'(\hat{b}) - \frac{F(\hat{b})}{q}\left(I_F'(\hat{b}) - 1\right).
\end{equation*}
Since $I_F$ is concave, it follows that
\begin{equation*}
I_F(\hat{b}) \leq \frac{\mu}{q} I_F'(\hat{b}) - \frac{F(\hat{b})}{q}\left(I_F'(\hat{b}) - 1\right).
\end{equation*}
Also, since we have that $0 \leq I_F'(\hat{b}) \leq 1$, using \eqref{inequalityb} yields
\begin{equation*}
I_F(\hat{b}) \leq \frac{\mu}{q}I_F'(\hat{b}) - \left(\frac{\psi(\hat{b})}{\psi'(\hat{b})} - \frac{\varphi_F(\hat{b})}{\varphi_F'(\hat{b})}\right) \left(I_F'(\hat{b}) - 1\right).
\end{equation*}
This inequality is equivalent to $g(\hat{b}) \leq h(\hat{b})$.
\end{proof}

The next result states that if $b^*$ is a solution to~\eqref{solution}, as in the previous proposition, then $\pi_{b^*}$ satisfies the conditions of the Verification Lemma.

\begin{prop}
If $I_F'(0) - I_F(0) \varphi_F'(0) > 1$ and if $b^* \in (0, \hat{b}]$ is a solution to~\eqref{solution}, then $V_{b^*} \in C^2(0, +\infty)$ is concave. 
\end{prop}

\begin{proof}
This proof relies heavily on the analytical properties of $\psi, \varphi_F$ and $I_F$ obtained in Lemma~\ref{AnalyticLemma} and Proposition~\ref{associate}.

First, let us show that $V_{b^*} \in C^2(0, +\infty)$. From Proposition~\ref{valuefunctionprop}, we deduce that
\begin{equation*}
V_{b^*}''(x) = \twopartdef{C_1(b^*) \psi''(x)}{0 < x < b^*,}{I_F''(x) + C_2(b^*)\varphi_F''(x)}{x > b^*.}
\end{equation*}
Therefore, since $\psi$, $\varphi_F$ and $I_F$ are twice continuously differentiable functions,  it is sufficient to show that $V_{b^*}''(b^*-) = V_{b^*}''(b^*+)$. We also have that $\psi$, $\varphi_F$ and $I_F$ are solutions to second-order ODEs, so it is equivalent to show that
\begin{multline*}
\mu\left[C_1(b^*) \psi'(b^*) - C_2(b^*) \varphi_F'(b^*) - I_F'(b^*)\right] \\
- q\left[C_1(b^*) \psi(b^*) - C_2(b^*) \varphi_F(b^*) - I_F(b^*)\right] \\
+ F(b^*)\left[C_2(b^*) \varphi_F'(b^*) - I_F'(b^*) - 1\right] = 0.
\end{multline*}
The statement follows from the fact that $V_{b^*}$ is continuously differentiable at $x = b^*$ and because $V_{b^*}'(b^*+) = 1$.

Now, let us show that $V_{b^*}$ is concave. Since $V_{b^*}'(b^*) = 1$, it follows directly that
\begin{equation}\label{c1}
C_1(b^*) = \frac{1}{\psi'(b^*)}
\end{equation}
and
\begin{equation}\label{c2}
C_2(b^*) = \frac{1 - I_F'(b^*)}{\varphi_F'(b^*)}.
\end{equation}
Using the analytical properties of $\psi, \varphi_F$ and $I_F$, it is clear that $C_2(b^*) \leq 0 < C_1(b^*)$. Since $b^* \leq \hat{b}$, we have that $\psi$ is concave on $(0, b^*)$, and so $V_{b^*}''(x) \leq 0$, for all $x \in (0, b^*)$. Finally, since $I_F$ is concave, $\varphi_F$ is convex, and $C_2(b^*) \leq 0$, we have that $V_{b^*}''(x) \leq 0$, for all $x \in (b^*,\infty)$. In other words, $V_{b^*}$ is concave.
\end{proof}


We are now ready to state the main result, which is a solution to the general control problem.
\begin{thm}\label{mainresult}
If $I_F'(0) - I_F(0)\varphi_F'(0) \leq 1$, then $\pi_0$ is an optimal strategy and the optimal value function is given by
\begin{equation*}
V(x) = I_F(x) - I_F(0) \varphi_F(x), \quad x \geq 0 .
\end{equation*}
If $I_F'(0) - I_F(0) \varphi_F'(0) > 1$, then $\pi_{b^*}$ is an optimal strategy, with $b^*$ a solution to \eqref{solution}, and the optimal value function is given by
\begin{equation*}
V(x) = \twopartdef{\frac{\psi(x)}{\psi'(b^*)}}{0 \leq x \leq b^*,}{I_F(x) + \frac{1 - I_F'(b^*)}{\varphi_F'(b^*)} \varphi_F(x)}{x \geq b^*.}
\end{equation*}
\end{thm}

\begin{proof}
All is left to justify is the expression for the optimal value function $V$ under the condition $I_F'(0) - I_F(0) \varphi_F'(0) > 1$. In that case, it suffices to use the general expression for $V_b$ obtained in Proposition~\ref{valuefunctionprop} together with the expressions for $C_1(b^*)$ and $C_2(b^*)$ given in Equations~\eqref{c1} and~\eqref{c2}.
\end{proof}

As announced, given an increasing and concave function $F$, and recalling that $\psi$ is independent of $F$ and always known explicitly (see Equation~\eqref{scale}), this solution to the control problem is explicit up to the computations of the functions $\varphi_F$ and $I_F$. It is interesting to note that these functions only depend on the dynamics of the nonlinear Ornstein-Uhlenbeck process $U$ as given in~\eqref{allin}; they are also solutions to ODEs.


\subsection{Solution to the problem with an affine bound}

If we choose $F(x) = R+Kx$, with $K > 0$ and $R \geq 0$, then $U$ is such that
\begin{equation*}
\mathrm{d}U_t = (\mu - R + KU_t) \mathrm{d}t + \sigma \mathrm{d}W_t ,
\end{equation*}
i.e., it is a \textit{standard} Ornstein-Uhlenbeck process. In this case, $\varphi_F$ is known explicitly (see \cite{borodin-salminen_2002}):
\begin{equation*}
\varphi_F(x) = \frac{H_K^{(q)}(x; \mu - R, \sigma)}{H_K^{(q)}(0; \mu - R, \sigma)}, \quad x \geq 0,
\end{equation*}
where, using the notation in \cite{renaud-simard_2021},
\begin{equation*}
H_K^{(q)}(x; m, \sigma) = \mathrm e^{K((x - m/K)^2/2\sigma^2}D_{-q/K}\left( \left(\frac{x - m/K}{\sigma}\right) \sqrt{2K}\right),
\end{equation*}
where $D_{- \lambda}$ is the parabolic cylinder function given by
\begin{equation*}
D_{- \lambda}(x) = \frac{1}{\Gamma(\lambda)} \mathrm e^{-x^2/4} \int_0^\infty t^{\lambda - 1} \mathrm e^{-xt-t^2/2} \mathrm{d}t, \quad x \in \R.
\end{equation*}

On the other hand, we can compute the expectation in the definition of $I_F$:
\begin{equation}\label{particularcase}
I_F(x) = \frac{K}{q+K}\left(x + \frac{\mu}{q}\right) + \frac{Rq}{q+K}.
\end{equation}

The following corollary is a generalization of both the results obtained in \cites{jeanblanc-shiryaev_1995,asmussen-taksar_1997} and in \cite{renaud-simard_2021} for a constant bound and a linear bound on control rates, respectively.
\begin{cor}\label{cor:renaud-simard}
Set $F(x)=R+Kx$,  with $K > 0$ and $R \geq 0$. Define $\Delta = -\frac{H_K^{(q)}(0; \mu-R, \sigma)}{H_K^{(q) \prime} (0; \mu-R, \sigma)}$. 

If $\Delta \geq \frac{K \mu}{q^2} + \frac{R}{q}$, then the mean-reverting strategy $\pi_0$ is an optimal strategy and the optimal value function is given, for $x \geq 0$, by
\begin{equation*}
V(x) = \frac{Kx}{q+K} - \left[ \frac{K}{q+K}\left(\frac{\mu}{q}\right) + \frac{Rq}{q+K} \right] \left(1- \frac{H_K^{(q)}(0; \mu - R, \sigma)}{H_K^{(q)}(0; \mu - R, \sigma)} \right) .
\end{equation*}

If $\Delta < \frac{K \mu}{q^2} + \frac{R}{q}$, then there exists a (unique) solution $b^* \in (0,\hat{b}]$ to
\begin{equation*}
\frac{\psi(b)}{\psi'(b)} - \left(b + \frac{\mu}{q}\right) - \frac{R}{K} = -\frac{q}{K}\left(\frac{\psi(b)}{\psi'(b)} - \frac{H_K^{(q)}(b; \mu - R, \sigma)}{H_K^{(q) \prime}(b; \mu - R, \sigma)}\right) ,
\end{equation*}
the mean-reverting strategy $\pi_{b^*}$ is an optimal strategy and the optimal value function is given by
\begin{equation*}
V(x) = \twopartdef{\frac{\psi(x)}{\psi'(b^*)}}{0 \leq x \leq b^*,}{\frac{K}{q+K}\left(x + \frac{\mu}{q}\right) + \frac{q}{q+K} \left(R+ \frac{H_K^{(q)}(x; \mu - R, \sigma)}{H_K^{(q)\prime}(b^*; \mu - R, \sigma)} \right)}{x \geq b^*.}
\end{equation*}
\end{cor}


\subsection{Solution to the problem with a capped linear bound}

If we choose $F(x) = \min(Kx, R)$, with $K, R > 0$, then $U$ is such that
\begin{equation*}
dU_t = (\mu - \min(KU_t, R))\mathrm{d}t + \sigma \mathrm{d}W_t,
\end{equation*} 
i.e., it is a Brownian motion with drift $\mu - R$ when $U_t > R/K$ and a standard Ornstein-Uhlenbeck process when $U_t < R/K$.

Note that $F$ is not differentiable at $x = R/K$, but there exists a decreasing sequence of differentiable and concave functions $F_n$ such that, for all $x \geq 0$, we have $F_n(x) \to F(x)$ when $n \to \infty$. By continuity, we have $\varphi_{F_n} \to \varphi_F$, and $I_{F_n} \to I_F$. Also, since $F$ is continuous, $I_F$ and $\varphi_F$ are (at least) twice continuously differentiable. 

We can compute $\varphi_F$ and $I_F$ explicitly, using the strong Markov property repeatedly:
\begin{equation*}
\varphi_F(x) = \twopartdef{B(x) + C(x) \varphi_F(R/K)}{0 \leq x \leq R/K,}{A(x) \varphi_F(R/K)}{x \geq R/K,}
\end{equation*}
and
\begin{equation*}
I_F(x) = \twopartdef{I_{K}(x) - D(x)\left(I_{K}(R/K) - I_F(R/K)\right)}{x \leq R/K,}{\frac{R}{q}\left(1 - A(x)\right) + A(x) I_F(R/K)}{x \geq R/K,}
\end{equation*}
where (see \cite{borodin-salminen_2002})
\begin{align*}
A(x) &= \esub{x}\left[\mathrm e^{-q \tau_{R/K}} \I{\tau_{R/K} < \infty}\right] \\
&= \tilde{\psi}(x - R/K) - \frac{2 q}{\sqrt{(\mu - R)^2 + 2\sigma^2 q} - (\mu - R)}\psi(x - R/K), \\
B(x) &= \esub{x}\left[\mathrm e^{-q \tau_0^F} \I{\tau_0^F < \tau_{R/K}^F}\right] = \frac{S\left(\frac{q}{K};\sqrt{\frac{2}{K}}\left(\frac{R-\mu}{\sigma}\right);\frac{x - \mu/K}{\sigma/\sqrt{2K}}\right)}{S\left(\frac{q}{K};\sqrt{\frac{2}{K}}\left(\frac{R-\mu}{\sigma}\right);\sqrt{\frac{2}{K}} \left(\frac{-\mu}{\sigma}\right)\right)}, \\
C(x) &= \esub{x}\left[\mathrm e^{-q \tau_{R/K}^F} \I{\tau_{R/K}^F < \tau_0^F}\right] = \frac{S\left(\frac{q}{K};\frac{x - \mu/K}{\sigma/\sqrt{2K}};\sqrt{\frac{2}{K}} \left(\frac{-\mu}{\sigma}\right)\right)}{S\left(\frac{q}{K};\sqrt{\frac{2}{K}}\left(\frac{R-\mu}{\sigma}\right);\sqrt{\frac{2}{K}} \left(\frac{-\mu}{\sigma}\right)\right)}, \\
D(x) &= \esub{x}\left[ \mathrm e^{-q \tau_{R/K}^F}\I{\tau_{R/K}^F < \infty}\right] = \frac{H_K^{(q)}(x; \mu, -\sigma)}{H_K^{(q)}(R/K; \mu, -\sigma)},
\end{align*}
where
\begin{equation*}
\tilde{\psi}(x) := 1 + \frac{2q}{\sigma^2}\int_0^x \psi(y) \mathrm{d}y
\end{equation*}
and
\begin{equation*}
S(\nu, x, y) := \frac{\Gamma(\nu)}{\pi} \mathrm e^{(x^2 + y^2)/4}\left(D_{-\nu}(-x) D_{-\nu}(y) - D_{-\nu}(x) D_{-\nu}(-y)\right),
\end{equation*}
and where
\begin{equation*}
I_K(x) := \frac{K}{q+K}\left(x + \frac{\mu}{q}\right).
\end{equation*}
Note that $I_K$ is the function $I_F$ when $F(x) = Kx$. See~\eqref{particularcase} when $R = 0$. 

Since $\varphi_F$ and $I_F$ are continuously differentiable, we can deduce that
\begin{equation*}
\varphi_F(R/K) = \frac{B'(R/K)}{A'(R/K) - C'(R/K)}
\end{equation*}
and
\begin{equation*}
I_F(R/K) = \frac{K/(q+K) - D'(R/K) I_{K}(R/K) + (R/q)A'(R/K)}{A'(R/K) - D'(R/K)}.
\end{equation*}


\begin{cor}
Set $F(x) = \min(Kx, R)$, with $K,R > 0$.

If $I_F'(0) - I_F(0) \varphi_F'(0) \leq 1$, then the generalized mean-reverting strategy $\pi_0$ is optimal and the optimal value function is given, for $x \geq 0$, by \begin{equation*}
V(x) = \twopartdef{I_{K}(x) - D(x)(I_{K}\left(\frac{R}{K}\right) - I_F\left(\frac{R}{K}\right)) - I_F(0)\left(B(x) + C(x) \varphi_F(\frac{R}{K})\right)}{x  \leq \frac{R}{K},}{\frac{R}{q} + A(x)\left(I_F\left(\frac{R}{K}\right) - I_F(0) \varphi_F\left(\frac{R}{K}\right) - \frac{R}{q}\right)}{x \geq \frac{R}{K}.}
\end{equation*} 

If $I_F'(0) - I_F(0) \varphi_F'(0) > 1$, then the generalized mean-reverting strategy $\pi_{b^*}$ is optimal, where $b^* \in (0, \hat{b}]$ is the unique solution to~\eqref{solution}, and the optimal value function is given by: if $b^* \leq \frac{R}{K}$, then
\begin{equation*}
V(x) = \threepartdef{\frac{\psi(x)}{\psi'(b^*)}}{x \leq b^*,}{I_{K}(x) - D(x)\left(I_{K}\left(\frac{R}{K}\right) - I_F\left(\frac{R}{K}\right)\right) - \frac{1 - I_F'(b^*)}{\varphi_F(b^*)}\left(B(x) + C(x) \varphi_F\left(\frac{R}{K}\right)\right)}{b^* \leq x \leq \frac{R}{K},}{\frac{R}{q} + A(x)\left(I_F\left(\frac{R}{K}\right) + \frac{1 - I_F'(b^*)}{\varphi_F'(b^*)}\varphi_F\left(\frac{R}{K}\right) - \frac{R}{q} \right)}{x \geq \frac{R}{K},}
\end{equation*}
and if $b^* > \frac{R}{K}$, then
\begin{equation*}
V(x) = \twopartdef{\frac{\psi(x)}{\psi'(b^*)}}{x \leq b^*,}{\frac{R}{q} + A(x)\left(I_F\left(\frac{R}{K}\right) + \frac{1 - I_F'(b^*)}{\varphi_F'(b^*)}\varphi_F\left(\frac{R}{K}\right) - \frac{R}{q} \right)}{x \geq b^*.}
\end{equation*}
\end{cor}

\section*{Appendix A. Proof of the Verification Lemma}

The second statement of the lemma is a direct consequence of the definition of the optimal value function. 

Now, let $\pi$ be an arbitrary admissible strategy. Applying Ito's Lemma to the continuous semi-martingale $(t, U_t^\pi)$, using the function $g(t,y) := \mathrm e^{-qt}V_{\pi^*}(y)$, we find
\begin{align*}
\mathrm e^{-q(t \wedge \tau_0^\pi)} & V_{\pi^*}(U_{t \wedge \tau_0^\pi}^\pi) \\
&= V_{\pi^*}(x) + \int_0^{t \wedge \tau_0^\pi} \mathrm e^{-qs} \left(\frac{\sigma^2}{2}V_{\pi^*}''(U_s^\pi) + \mu V_{\pi^*}'(U_s^\pi) - qV_{\pi^*}(U_s^\pi) - l_s^\pi V_{\pi^*}'(U_s^\pi) \right) \mathrm{d}s \\
& \qquad + \int_0^{t \wedge \tau_0^\pi} \sigma \mathrm e^{-qs} V_{\pi^*}'(U_s^\pi) \mathrm{d}W_s.
\end{align*}

Taking expectations on both sides, we find
\begin{align*}
V_{\pi^*}(x) &= \esub{x}\left[ \mathrm e^{-q(t \wedge \tau_0^\pi)} V_{\pi^*}(U_{t \wedge \tau_0^\pi}^\pi)\right] \nonumber \\
& \quad - \esub{x}\left[\int_0^{t \wedge \tau_0^\pi} \mathrm e^{-qs} \left(\frac{\sigma^2}{2}V_{\pi^*}''(U_s^\pi) + \mu V_{\pi^*}'(U_s^\pi) - qV_{\pi^*}(U_s^\pi) - l_s^\pi V_{\pi^*}'(U_s^\pi) \right) \mathrm{d}s\right] \nonumber \\
&\geq \esub{x}\left[ \mathrm e^{-q(t \wedge \tau_0^\pi)} V_{\pi^*}(U_{t \wedge \tau_0^\pi}^\pi)\right] + \esub{x} \left[\int_0^{t \wedge \tau_0^\pi} \mathrm e^{-qs} l_s^\pi \mathrm{d}s\right],
\end{align*}
where the inequality is obtained directly from the fact that $V_{\pi^*}$ satisfies~\eqref{verificationsup}. Note also that the expectation of the stochastic integral is $0$ because $V_{\pi^*}'$ is bounded.  Letting $t \to \infty$, we get that $V_{\pi^*}(x) \geq V_\pi(x)$, for all $x > 0$.
\section*{Acknowledgements}

Funding in support of this work was provided by a Discovery Grant from the Natural Sciences and Engineering Research Council of Canada (NSERC) and a PhD Scholarship from the Fonds de Recherche du Québec - Nature et Technologies  (FRQNT).

%
%
\bibliographystyle{abbrv}
\bibliography{references_de-finetti}
%
%
\end{document}